\documentclass[journal]{ieeeconf}
\usepackage{color}
\usepackage{lineno}
\usepackage[cmex10]{amsmath}
\usepackage{amscd,amsfonts,amsmath,amssymb,mathrsfs,pifont,stmaryrd,tipa}
\usepackage{array,cases,dsfont,graphicx,texdraw}
\usepackage{wrapfig}
\usepackage{graphics} 
\usepackage{epsfig} 
\usepackage{mathbbold}
\usepackage{stfloats}
\usepackage{subfig}
\usepackage{hyperref}
\newtheorem{Remark}{Remark}
\newtheorem{Corollary}{Corollary}

\newtheorem{Problem}{Problem}
\newenvironment{Proof}{\noindent{\em Proof:\/}}{\hfill $\Box$\par}
\newtheorem{Theorem}{Theorem}

\newtheorem{Assumption}{Assumption}

\input{tcilatex}                                                          

\title{\LARGE \bf
Distributed Nash Equilibrium Seeking by A Consensus Based Approach}

\author{Maojiao Ye
        and Guoqiang Hu 
\thanks{M. Ye and G. Hu are with the School of Electrical and Electronic Engineering, Nanyang
Technological University, 639798, Singapore  (Email: mjye@ntu.edu.sg,gqhu@ntu.edu.sg).}
\thanks{This work was supported by Singapore Economic Development Board under EIRP grant S14-1172-NRF EIRP-IHL.}}

\begin{document}

\maketitle
\thispagestyle{empty}
\pagestyle{empty}

\begin{abstract}
In this paper, Nash equilibrium seeking among a network of players is considered. Different from many existing works on Nash equilibrium seeking in non-cooperative games, the players considered in this paper cannot directly observe the actions of the players who are not their neighbors. Instead, the players are supposed to be capable of communicating with each other via an undirected and connected communication graph. By a synthesis of a leader-following consensus protocol and the gradient play, a distributed Nash equilibrium seeking strategy is proposed for the non-cooperative games. Analytical analysis on the convergence of the players' actions to the Nash equilibrium is conducted via Lyapunov stability analysis. For games with non-quadratic payoffs, where multiple isolated Nash equilibria may coexist in the game, a local convergence result is derived under certain conditions. Then,
a stronger condition is provided to derive a non-local convergence result for the non-quadratic games. For quadratic games, it is shown that the proposed seeking strategy enables the players' actions to converge to the Nash equilibrium globally under the given conditions. Numerical examples are provided to verify the effectiveness of the proposed seeking strategy.
\end{abstract}
\begin{keywords}
Nash equilibrium; gradient play; leader-following consensus; neighboring communication
\end{keywords}

\section{INTRODUCTION}

The past decade witnessed the penetration of game theory into various research fields including biology, economy, computer science, just to name a few.  With the development of game theory, Nash equilibrium seeking in non-cooperative games emerges to be of both theoretical significance and practical relevance
(e.g., see \cite{Frihauf12}-\cite{NekoueiTSG15} and the references therein).

The authors in \cite{PozoTPS11} formulated pure strategy Nash equilibria seeking as a mixed-integer linear programming problem for pool-based electricity market.
Gradient play was leveraged for finding differential Nash equilibria in continuous games in \cite{RaltiffAAC13}. Dynamic fictitious play and gradient play were exploited in \cite{ArslanCDC01} for a continuous-time form of repeated matrix games. Policy evaluation and policy improvement were utilized for the computation of the Nash equilibrium in differential graphical games \cite{VamvoudakisAT12}. The discrete-time stochastic algorithm developed in \cite{PovedaACC15} allows the players to take actions in both simultaneous and asynchronous fashions. Based on the state-based potential games, the authors in \cite{LiTAC14} considered game design for distributed optimization problems, in which a distributed process was proposed to obtain the equilibrium. By utilizing the saddle point dynamics, convergence to the Nash equilibria of a two-network zero-sum game was derived in \cite{GhaersifardAT13}. Switching communications were considered for the two-network zero-sum games in \cite{YOUTAC15}.
Nash equilibrium seeking in generalized convex games was considered in \cite{zhucdc}-\cite{Schiro13}.
The authors in \cite{zhucdc} and \cite{zhu1} solved the generalized convex game by a discrete-time distributed algorithm and Lemke's method was adapted for the computation of generalized Nash equilibrium in convex games with quadratic payoffs subject to linear constraints in \cite{Schiro13}.
Besides, extremum seeking based approaches were proposed to seek for the Nash equilibrium (e.g., see \cite{Frihauf12} and \cite{StankoviTAC12}-\cite{DurrAT13}). 
These methods vary from integrator-type extremum seeking \cite{Frihauf12}, discrete-time extremum seeking \cite{StankoviTAC12}, stochastic extremum seeking \cite{Liusiam11}, Shahshahani gradient-like extremum seeking \cite{Povedaat15}, to Lie bracket approximation based extremum seeking \cite{DurrAT13}, etc.
A common characteristic of these methods is that no explicit model information is required for the implementation of the methods.

In non-cooperative games, the players' payoff functions are determined by the players' own actions, together with the other players' actions \cite{RosenEco65}. Hence, a body of the existing works require the players' observations over their opponents' actions to search for the Nash equilibrium. However, full communication is impractical in many engineering systems (e.g., multi-agent systems, ad hoc networks) \cite{SalehisadaghianiCDC14}.
Motivated by the penetration of the game theoretic approaches into cooperative control and distributed optimization problems in engineering systems where full communication is not available (see, e.g.,  \cite{VamvoudakisAT12,LiTAC14,MardenSMC09,TanakaCDC13}), this paper addresses Nash equilibrium seeking under local communication network, i.e., the players communicate with their neighbors only.

To solve games with limited information, the main idea of this paper is to utilize a consensus protocol to broadcast local information. Consensus problems have been extensively investigated in the existing literature (e.g., see \cite{CSM07}-\cite{Menon14} \cite{Andreasson14,Hug2015}).
For instance, a class of consensus controllers were proposed for networked dynamical systems in \cite{Andreasson14} and a consensus based approach was studied in  \cite{Hug2015} for distributed coordination of the generation, load and storage devices in a microgrid.
In particular, leader-following consensus concerns with the synchronization of the agents' states to a common value, which is equal to the reference signal provided by the leader \cite{HuSCL12}. The proposed Nash equilibrium seeking strategy is based on an adaptation of a leader-following consensus protocol and the gradient play. More specifically, each agent acts as a virtual leader to provide its action as the reference signal and the agents generate their estimates on the players' actions by utilizing a leader-following consensus protocol. Based on the estimates, the gradient play is implemented for each player.


\textbf{Related works:} Two-network zero sum games were investigated in \cite{GhaersifardAT13} and \cite{YOUTAC15} under communication graphs. Distributed learning for games on communication graph was investigated in \cite{Swenson12} and \cite{Swenson15} for large-scale multi-agent games by an adaptation of the fictitious play. Though the communication setting in \cite{Swenson15} was similar to the presented work, the authors considered games with identical permutation-invariant utilities or exact potential games with a permutation-invariant potential function, which are distinct from the games considered in this paper. In \cite{KoshalCDC12}, the authors considered Nash equilibrium seeking for aggregative games on graph, where each player's payoff function depends on their own actions and an aggregate of all the players' actions. A discrete-time gossip algorithm was proposed to solve it. A similar problem was addressed in \cite{YE15} for the energy consumption game in smart grids. The continuous-time method in \cite{YE15} was based on a dynamic average consensus protocol and the primal-dual dynamics. The idea on solving games without using full information from all the players was then generalized in \cite{SalehisadaghianiCDC14}, where the players' payoff functions depend on the players' actions in a more general manner. The Nash equilibrium was characterized by a variational inequality and a gossip-based algorithm was proposed. The players were equipped with a waking clock and they updated their actions asynchronously to reach the Nash equilibrium. Different from \cite{SalehisadaghianiCDC14}, we consider the Nash equilibrium seeking problem in a deterministic and continuous-time scenario. The continuous-time algorithm activates the powerful analysis tools in control theory to solve games. Compared with the existing works, the main contributions of the paper are summarized as follows.

\begin{enumerate}
  \item Nash equilibrium seeking for non-cooperative games, where the players have no direct access to the actions of the players who are not their neighbors, is investigated in this paper. Based on a leader-following consensus protocol and the gradient play, a Nash equilibrium seeking strategy is designed. In the proposed algorithm, the players only need to communicate with their neighbors on their estimates of the players' actions. Avoiding full communication among the players broadens the applicability of game theory to engineering systems where only local communication is attainable.
  \item The convergence of the players' actions to the Nash equilibrium by utilizing the proposed seeking strategy is analytically explored. Based on the Lyapunov stability analysis, it is shown that the proposed method enables the players' actions to converge to the Nash equilibrium under certain conditions. Non-quadratic games are firstly investigated followed by quadratic games.
\end{enumerate}

The rest of the paper is organized as follows. The problem is formulated in Section \ref{prob_form} and the main results are presented in Section \ref{main_res_ne}. In the main result part, non-quadratic games are firstly investigated followed by quadratic games. Numerical examples are presented in Section \ref{numer_example} to verify the effectiveness of the proposed method. Conclusions are given in Section \ref{concl_a}.

\section{Problem Formulation}\label{prob_form}
\begin{Problem} Consider a game with $N$ players. The set of players is denoted by $\mathcal{N}=\{1,2,\cdots,N\}.$ The payoff function of player $i$ is $f_i(\mathbf{x}),$ where $\mathbf{x}=[x_1,x_2,\cdots,x_N]^T\in R^N$ is the vector of the players' actions and $x_i\in R$ is the action of player $i$. Suppose that if player $j$ is not a neighbor of player $i$, then, player $i$ has no direct access to player $j$'s action.  Given that the Nash equilibrium of the game exists, design a Nash equilibrium seeking strategy that can be adopted by the players to learn the Nash equilibrium of the non-cooperative game.
\end{Problem}

For the convenience of the readers, the definition of the Nash equilibrium is given below.

\emph{Nash equilibrium} is an action profile
on which no player can gain more payoff by unilaterally changing its own
action, i.e., an action profile $\mathbf{x}^*=(x_{i}^{\ast},\mathbf{x}_{-i}^{\ast})$ is the
Nash equilibrium if \cite{NASH51}
\begin{equation}
 f_{i}(x_{i}^{\ast},\mathbf{x}_{-i}^{\ast})\geq f_{i}(x_{i},\mathbf{x}_{-i}^{\ast}),\forall i\in \mathcal{N},
 \end{equation}
 where $\mathbf{x}_{-i}=[x_1,x_2,\cdots,x_{i-1},x_{i+1},\cdots,x_N]^T.$ Note that $f_i(\mathbf{x})$ and $\mathbf{x}$ might alternatively be written as $f_i(x_i,\mathbf{x}_{-i})$ and $(x_i,\mathbf{x}_{-i})$, respectively, in this paper.

\begin{Remark}
The objective for the players in the game is different from the objective of the agents in distributed optimization problems. Given the other players' actions, the players in the game intend to maximize their own payoffs by adjusting their own actions. In contrast, the agents engaged in distributed optimization problems collaboratively maximize the sum of the agents' objective functions, i.e.,
\begin{equation}
\text{max}_{\mathbf{x}} \sum_{i=1}^N f_i(\mathbf{x}),
\end{equation}
where $f_i(\mathbf{x})$ is the objective function of agent $i$ and $\mathbf{x}$ is the vector of the decision variables (see, e.g., \cite{YECST}).
\end{Remark}

The following assumptions on the communication graph (see Section \ref{gra} for the related definitions) and the payoff functions will be utilized in the rest of the paper.
\begin{Assumption}\label{ASS_comm}
The players can communicate with each other via an undirected and connected communication graph.
\end{Assumption}
\begin{Assumption}\label{ASS_func_1}
The players' payoff functions $f_i(\mathbf{x}),\forall i\in \mathcal{N}$  are $\mathcal{C}^2$ functions.
\end{Assumption}

\section{Main Results}\label{main_res_ne}

In this section, a distributed Nash equilibrium seeking strategy will be proposed based on a leader-following consensus protocol and the gradient play. Non-quadratic games are firstly considered followed by quadratic games.

\emph{Strategy design:}
Noticing that the players have no direct access to the actions of the players who are not their neighbors, we suppose that each player generates estimates on the players' actions. Let $\mathbf{y}_i=[y_{i1},y_{i2},\cdots,y_{iN}]^T\in R^N$ denote player $i$'s estimates on the players' actions and $y_{ij}$ is player $i$'s estimate on player $j$'s action. Then, the action of player $i$ is updated according to
\begin{equation}\label{cons_xx}
\dot{x}_{i}=k_i \frac{\partial f_i}{\partial x_{i}}(\mathbf{y}_i),i\in \mathcal{N},
\end{equation}
where $k_i=\delta \bar{k}_i$, $\delta$ is a small positive parameter and $ \bar{k}_{i}$ is a positive, fixed parameter. Furthermore, $y_{ij}, \forall i,j\in \mathcal{N}$ are generated by
\begin{equation}\label{NE_ave_cons_1}
\dot{y}_{ij}=-\left(\sum_{k=1}^{N}a_{ik}(y_{ij}-y_{kj})+a_{ij}(y_{ij}-x_j)\right),
\end{equation}
where $a_{ij}$ is the element on the $i$th row and $j$th column of the adjacency matrix of the communication graph.

\emph{Insight into the strategy design:}
Let $\tau=\delta t$. Then, at the $\tau$-time scale,
\begin{equation}
\begin{aligned} \frac{dx_i}{d\tau}&=\bar{k}_i\frac{\partial f_i}{\partial x_i}(\mathbf{y}_i),\\
\delta\frac{dy_{ij}}{d\tau}&=-\left(\sum_{k=1}^Na_{ik}(y_{ij}-y_{kj})+a_{ij}(y_{ij}-x_j)\right),
\end{aligned}
\end{equation}
$\forall i,j\in \mathcal{N}.$ Define $y_{ij}^q$ as the quasi-steady state of $y_{ij},$ for all  $i,j\in \mathcal{N},$ i.e.,   $y_{ij}^q$ for $i,j\in \mathcal{N}$ satisfy
\begin{equation}
-(\sum_{k=1}^Na_{ik}(y_{ij}^q-y_{kj}^q)+a_{ij}(y_{ij}^q-x_j))=0, \forall i,j\in \mathcal{N}.
\end{equation}
Then, $y_{ij}^q=x_j, \forall i,j\in \mathcal{N}$ as the communication graph is undirected and connected.

Letting $\delta=0$ ``freezes" $y_{ij},\forall i,j\in \mathcal{N}$ on the quasi-steady state  on which $y_{ij}=y_{ij}^q=x_j,\forall i,j\in \mathcal{N}.$ Then, the reduced-system is
\begin{equation}
\frac{dx_i}{d\tau}=\bar{k}_i\frac{\partial f_i}{\partial x_i}(\mathbf{x}),\forall i\in \mathcal{N}
\end{equation}
by which the convergence to the Nash equilibrium can be derived under certain conditions (see, e.g., \cite{Frihauf12} and \cite{RaltiffAAC13}).

%

\subsection{Games with Non-quadratic Payoffs}\label{non_ga}

In the following, we show that the seeking strategy in \eqref{cons_xx} and \eqref{NE_ave_cons_1} enables the players' actions to converge to the Nash equilibrium. To facilitate the subsequent analysis, the following assumptions are made.

\begin{Assumption}\label{ASS_2}
There exists at least one, possibly multiple Nash equilibria $\mathbf{x}^*=[x_1^*,x_2^*,\cdots,x_N^*]^T$ such that for all $i\in \mathcal{N},$
\begin{equation*}
\frac{\partial f_i}{\partial x_i}(\mathbf{x}^*)=0,\frac{\partial^2 f_i}{\partial x_i^2}(\mathbf{x}^*)<0.
\end{equation*}
\end{Assumption}

\begin{Assumption}\label{ASS_2_1}
The matrix
\begin{equation*}
B=\small{\left[
                \begin{array}{cccc}
                  \frac{\partial ^2f_1}{\partial x_1^2}(\mathbf{x}^*) & \frac{\partial ^2f_1}{\partial x_1 \partial x_2}(\mathbf{x}^*) & \cdots &  \frac{\partial ^2f_1}{\partial x_1 \partial x_N}(\mathbf{x}^*) \\
                  \frac{\partial ^2f_2}{ \partial x_2 \partial x_1}(\mathbf{x}^*) & \frac{\partial ^2f_2}{\partial x_2^2}(\mathbf{x}^*) &  &  \frac{\partial ^2f_2}{ \partial x_2 \partial x_N}(\mathbf{x}^*) \\
                  \vdots &  & \ddots &  \\
                  \frac{\partial ^2f_N}{\partial x_N \partial x_1 }(\mathbf{x}^*) &\frac{\partial ^2f_N}{\partial x_N \partial x_2 }(\mathbf{x}^*)  &  & \frac{\partial ^2f_N}{\partial x_N^2}(\mathbf{x}^*) \\
                \end{array}
              \right]},
\end{equation*}
is strictly diagonally dominant, i.e., $|\frac{\partial ^2f_i}{\partial x_i^2}(\mathbf{x}^*)|>\sum_{j=1,j\neq i}^N |\frac{\partial ^2f_i}{\partial x_i\partial x_j}(\mathbf{x}^*)|,\forall i\in \mathcal{N}$  \cite{Horn85}.
\end{Assumption}

\begin{Remark}
Assumptions \ref{ASS_2}-\ref{ASS_2_1} are adoptions of Assumptions 4.3-4.4 in \cite{Frihauf12}.  By Theorem 2 in \cite{RaltiffAAC13}, the Nash equilibria that satisfy Assumptions \ref{ASS_2}-\ref{ASS_2_1} are isolated. The objective of this paper is to design a Nash equilibrium seeking strategy that can be utilized by the players to learn the Nash equilibrium of the non-cooperative games. The characterizations on the existence, uniqueness and isolation issues of the Nash equilibria are beyond the scope of the paper. The readers are referred to \cite{RaltiffAAC13,RosenEco65} for the related explorations.
\end{Remark}

For notational convenience, let $\text{diag}\{p_{ij}\},p_{ij}\in R, i,j\in \mathcal{N}$ be a diagonal matrix whose diagonal elements are $p_{11},p_{12},\cdots,p_{1N},p_{21},\cdots,p_{NN},$ successively. Similarly, $\text{diag}\{p_i\},p_i\in R, i\in \mathcal{N}$ denotes an $N\times N$ dimensional diagonal matrix whose $i$th diagonal element is $p_i$. Moreover, define $\frac{\partial G(\mathbf{x})}{\partial \mathbf{x}}=\left[\frac{\partial f_1(\mathbf{x})}{\partial x_1},\frac{\partial f_2(\mathbf{x})}{\partial x_2},\cdots,\frac{\partial f_N(\mathbf{x})}{\partial x_N}\right]^T,$ $\frac{\partial G}{\partial \mathbf{x}}(\mathbf{y})=\left[\frac{\partial f_1}{\partial x_1}(\mathbf{y}_1),\frac{\partial f_2}{\partial x_2}(\mathbf{y}_2),\cdots,\frac{\partial f_N}{\partial x_N}(\mathbf{y}_N)\right]^T,$ $\mathbf{y}=[\mathbf{y}_1^T,\mathbf{y}_2^T,\cdots,\mathbf{y}_N^T ]^T,$ $h_1(\mathbf{x})=-[a_{11}x_1,a_{12}x_2,\cdots,a_{1N}x_N,a_{21}x_1,a_{22}x_2,\cdots, a_{NN}x_N]^T,$ and $B_0=\text{diag}\{a_{ij}\}, i,j \in \mathcal{N}.$ Then, the concatenated form of \eqref{NE_ave_cons_1} is
\begin{equation}
\dot{\mathbf{y}}=-\left((L\otimes I_{N\times N}+B_0)\mathbf{y}+h_1(\mathbf{x})\right),
\end{equation}
where $L$ is the Laplacian matrix of the communication graph, $I_{N\times N}$ denotes an $N\times N$ dimensional identity matrix and $\otimes$ is the Kronecker  product. Note that $-(L\otimes I_{N\times N}+B_0)$ is Hurwitz as the communication graph is undirected and connected.

\begin{Theorem}\label{res_1}
Suppose that Assumptions \ref{ASS_comm}-\ref{ASS_func_1} hold, and the agents update their actions according to \eqref{cons_xx}-\eqref{NE_ave_cons_1}. Then, for each $\mathbf{x}^*$ that further satisfies Assumptions \ref{ASS_2}-\ref{ASS_2_1}, there exists a positive constant $\delta^*$ such that for each $\delta\in (0,\delta^*)$, $(\mathbf{x}^*,\mathbf{1}_N\otimes \mathbf{x}^*)$ is exponentially stable.
\end{Theorem}
\begin{Proof}
See Section \ref{res_1_p} for the proof.
\end{Proof}

\begin{Remark}
Theorem \ref{res_1} can be derived by utilizing singular perturbation analysis (see e.g., \cite{Frihauf12}\cite{KHAIL02}). The detailed Lyapunov stability analysis is conducted in the proof for the convenience of the subsequent non-local convergence analysis.  From the proof, it can be seen that the convergence result depends on the convergence of the players' actions to each isolated Nash equilibrium under the gradient play, i.e.,
\begin{equation}
\frac{dx_i}{dt}=\bar{k}_i\frac{\partial f_i}{\partial x_i}(\mathbf{x}),\forall i\in \mathcal{N},
\end{equation}
which is ensured if the matrix $\bar{\mathbf{k}}B,$ where $\bar{\mathbf{k}}=\text{diag}\{\bar{k}_i\},i\in \mathcal{N},$  is Hurwitz. Hence, Assumption \ref{ASS_2_1} is conservative to derive the result (see also \cite{Frihauf12} for the argument).

Note that Assumption \ref{ASS_2_1} is different from the condition provided in Proposition 2 of \cite{RaltiffAAC13}. In \cite{RaltiffAAC13}, the gradient play is governed by
\begin{equation}\label{gradie_p}
\frac{dx_i}{dt}=\frac{\partial f_i}{\partial x_i}(\mathbf{x}),\forall i\in \mathcal{N}.
\end{equation}
Linearizing \eqref{gradie_p} at the Nash equilibrium point, it can be derived that the Nash equilibrium is exponentially stable under \eqref{gradie_p} if $B$ is Hurwitz \cite{RaltiffAAC13}.
\end{Remark}

In Theorem \ref{res_1}, a local convergence result to each isolated Nash equilibrium that satisfies the given conditions is presented. In the following, non-local convergence results are investigated under stronger conditions.
\begin{Assumption}\label{Ass_4}
Each player's payoff function $f_i(x_i,\mathbf{x}_{-i})$ is concave in $x_i, \forall i\in\mathcal{N}$. Furthermore, for $\mathbf{x},\mathbf{z}\in R^N,$
\begin{equation}
(\mathbf{x}-\mathbf{z})^T\left(\frac{\partial G(\mathbf{x})}{\partial \mathbf{x}}-\frac{\partial G(\mathbf{z})}{\partial \mathbf{z}}\right)\leq -m||\mathbf{x}-\mathbf{z}||^2,
\end{equation}
where $m>0$ is a constant.
\end{Assumption}
\begin{Remark}
By this assumption, it can be derived that the Nash equilibrium of the game is unique. In addition, the stationary condition $\frac{\partial G(\mathbf{x})}{\partial \mathbf{x}}=\mathbf{0}_N$, where $\mathbf{0}_N$ denotes an $N$ dimensional column vector composed of $0,$ is a sufficient and necessary condition for $\mathbf{x}=\mathbf{x}^*,$ which indicates that
\begin{equation}
(\mathbf{x}-\mathbf{x}^*)^T\frac{\partial G(\mathbf{x})}{\partial \mathbf{x}}\leq -m||\mathbf{x}-\mathbf{x}^*||^2,
\end{equation}
for $\mathbf{x}\in R^N.$ Similar to \cite{RosenEco65}, the uniqueness of the equilibrium point can be derived by supposing that
\begin{equation}\label{rosen_c}
(\mathbf{x}-\mathbf{z})^T\left(\frac{\partial G(\mathbf{x})}{\partial \mathbf{x}}-\frac{\partial G(\mathbf{z})}{\partial \mathbf{z}}\right)<0,
\end{equation}
for all distinct $\mathbf{x},\mathbf{z}$ in the considered domain. Assumption \ref{Ass_4} is slightly stronger than \eqref{rosen_c} to derive exponential stability of the Nash equilibrium under the proposed seeking strategy.
\end{Remark}
\begin{Theorem}\label{corol_2}
Suppose that Assumptions \ref{ASS_comm}, \ref{ASS_func_1} and \ref{Ass_4} are satisfied and the players update their actions according to \eqref{cons_xx}-\eqref{NE_ave_cons_1}. Then, for each positive constant $\Delta$, there is a positive constant $\delta^*(\Delta)$ such that for each $\delta\in (0,\delta^*(\Delta))$,
 $(\mathbf{x},\mathbf{y})$ converges exponentially to $(\mathbf{x}^*,\mathbf{1}_N\otimes \mathbf{x}^*)$ for every $||[(\mathbf{x}(0)-\mathbf{x}^*)^T,(\mathbf{y}(0)-\mathbf{1}_N\otimes \mathbf{x}^*)^T]^T||\leq \Delta.$
\end{Theorem}
\begin{proof}
See Section \ref{corol_2_p} for the proof.
\end{proof}

\begin{Corollary}\label{corol_2_add}
Suppose that Assumptions \ref{ASS_comm}, \ref{ASS_func_1} and \ref{Ass_4} are satisfied, the players update their actions according to \eqref{cons_xx}-\eqref{NE_ave_cons_1} and the functions $\frac{\partial f_i(\mathbf{x})}{\partial x_i},\forall i\in \mathcal{N}$ are globally Lipschitz. Then, there is a positive constant $\delta^*$ such that for each $\delta\in (0,\delta^*)$,
$(\mathbf{x}^*,\mathbf{1}_N\otimes \mathbf{x}^*)$  is globally exponentially stable.
\end{Corollary}
\begin{proof}
See Section \ref{corol_2_add_p} for the proof.
\end{proof}
\begin{Remark}
A typical example that satisfies Assumption \ref{Ass_4} is a potential game\footnote{Given that the players' payoff functions are continuously differentiable, the game is a \emph{potential game} if there exists a function $F(\mathbf{x})$ such that
\begin{equation}\label{pote_ga_de}
\frac{\partial f_i(x_i,\mathbf{x}_{-i})}{\partial x_i}=\frac{\partial F(x_i,\mathbf{x}_{-i})}{\partial x_i},
\end{equation}
$\forall i \in \mathcal{N}.$ Furthermore, the function $F$ is the potential function \cite{YETC15}.} in which the potential function is strongly concave\footnote{ A differentiable function $f$ is \emph{strongly convex} if the following inequality holds for all points $x,y$ in its domain:
\begin{equation}
(x-y)^T\left(\frac{\partial f(x)}{\partial x}-\frac{\partial f(y)}{\partial y}\right)\geq m||x-y||^2,
\end{equation}
for some parameter $m>0.$ A function $f$ is \emph{strongly concave} if $-f$ is strongly convex \cite{Bertsekas}.
}.
\end{Remark}

\subsection{Quadratic Games}
In this section, quadratic games are considered. Suppose that player $i$'s payoff function is
\begin{equation}
f_i(\mathbf{x})=\frac{1}{2}\sum_{j=1}^{N}\sum_{k=1}^{N}h^i_{jk}x_jx_k+\sum_{j=1}^{N} v^i_jx_j+g_i,
\end{equation}
where $h_{jk}^i,$ $v^i_j$, $g_i$ are the coefficients of the quadratic terms,  monomial terms and constant terms, respectively. Furthermore, $h_{ii}^i<0, h_{jk}^i=h_{kj}^i, \forall i,j,k\in \mathcal{N}$.

\begin{Assumption}\label{Assum_5}
The matrix
\begin{equation*}
H=\left[
  \begin{array}{cccc}
    h_{11}^1 & h_{12}^1 & \cdots & h_{1N}^1 \\
    h_{21}^2 & h_{22}^2 & \cdots & h_{2N}^2 \\
    \vdots &  & \ddots &  \\
    h_{N1}^N & h_{N2}^N & \cdots & h_{NN}^N \\
  \end{array}
\right],
\end{equation*}
is strictly diagonally dominant, i.e., $|h_{ii}^i|>\sum_{j=1,j\neq i}^N|h_{ij}^i|,\forall i\in \mathcal{N}$.
\end{Assumption}

\begin{Remark}
By Assumption \ref{Assum_5}, the Nash equilibrium of the quadratic games exists and is unique. Moreover, the unique Nash equilibrium is given by $\mathbf{x}^*=-H^{-1}v,$ where $v=[v_1^1,v_2^2,\cdots,v_N^N]^T$ \cite{Frihauf12,Basar99}.
\end{Remark}

\begin{Corollary}\label{qua_a}
Suppose that Assumptions \ref{ASS_comm}, \ref{ASS_func_1} and \ref{Assum_5} hold, and the agents update their actions according to \eqref{cons_xx}-\eqref{NE_ave_cons_1}. Then, there exists a positive constant $\delta^*$ such that for each $\delta\in (0,\delta^*)$, $(\mathbf{x}^*,\mathbf{1}_N\otimes \mathbf{x}^*)$ is globally exponentially stable.
\end{Corollary}

\begin{Proof}
See Section \ref{qua_a_p} for the proof.
\end{Proof}

\begin{Corollary}\label{qua_ca}
Suppose that Assumptions \ref{ASS_comm}-\ref{ASS_func_1} are satisfied, matrix $H$ is Hurwitz and the agents update their actions according to \eqref{cons_xx}-\eqref{NE_ave_cons_1}. Then, there exists a positive constant $\delta^*$ such that for each $\delta\in (0,\delta^*)$, the Nash equilibrium is globally exponentially stable given that the quadratic game is a potential game.
\end{Corollary}
\begin{Proof}
See Section \ref{qua_ca_p} for the proof.
\end{Proof}

\begin{Remark}\label{remar_gain}
The presented results still hold if the estimation dynamics in \eqref{NE_ave_cons_1} is changed to
\begin{equation}
\dot{y}_{ij}=-m_{ij}\left(\sum_{k=1}^{N}a_{ik}(y_{ij}-y_{kj})+a_{ij}(y_{ij}-x_j)\right),
\end{equation}
$i,j\in \mathcal{N},$ where $m_{ij}$ is a fixed positive constant as the matrix $-\mathbf{m}(L\otimes I_{N\times N}+B_0),$ where $\mathbf{m}=\text{diag}\{m_{ij}\},i,j\in \mathcal{N},$ is Hurwitz.
\end{Remark}

\begin{Remark}
The proposed seeking strategy can be leveraged to seek for the Nash equilibrium for games where each player's payoff function depends on its own action and all the other players' actions. However, for aggregative games in which each player's payoff function depends on its own action and an aggregate of the all the player's actions, it is not necessary for each player to estimate all the players' actions. Instead, a dynamic average consensus protocol can be utilized to estimate the aggregate of the players' actions to reduce the computation cost (see e.g., \cite{YE15}). Moreover, for games where the players' payoff functions depend on a subset of the players' actions, an interference graph can be introduced to describe the interactions among the players (see e.g.,  \cite{Tekin12}). To further reduce the computation cost, the adaptation of the proposed seeking strategy for games on interference graph will be considered in future work.
\end{Remark}

\section{Numerical Examples} \label{numer_example}
In this section, three games with a network of $5$ players are considered. The communication graph for the players in the following examples is depicted in Fig. \ref{commu_graph_NE}.
\begin{figure}
\begin{center}
\scalebox{0.35}{\includegraphics[12,508][284,788]{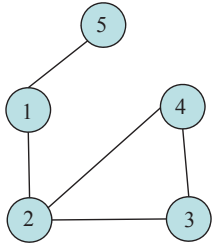}}
\caption{Communication graph for the players in the numerical examples.}\label{commu_graph_NE}
\end{center}
\end{figure}

\subsection{ Non-quadratic games}

\begin{example}\label{exam_1}
The players' payoff functions are
$f_1(\mathbf{x})=-x_1^3+3x_1x_2, f_2(\mathbf{x})=-(-2x_1+4x_2+\frac{1}{2}x_4+x_5)^2+48x_2,f_3(\mathbf{x})=-(x_1+4x_3-x_4-x_5)^2,f_4(\mathbf{x})=-(2x_1+4x_3+8x_4-x_5)^2,f_5(\mathbf{x})=-(x_1+4x_3+8x_4+17x_5)^2,$
for players 1-5, respectively.

\end{example}

Letting $\frac{\partial f_i(\mathbf{x})}{\partial x_i}=0,\forall i\in \{1,2,\cdots,5\}$ gives $\mathbf{x}=[-1,1,\frac{19}{72},\frac{1}{9},-\frac{1}{18}]^T$ or $\mathbf{x}=[\frac{3}{2},\frac{9}{4},-\frac{19}{48},-\frac{1}{6},\frac{1}{12}]^T.$ However, as $\frac{\partial ^2 f_1(\mathbf{x})}{\partial x_1^2}>0$ for $x_1<0,$ the point $\mathbf{x}=[-1,1,\frac{19}{72},\frac{1}{9},-\frac{1}{18}]^T$ does not satisfy Assumptions \ref{ASS_2}-\ref{ASS_2_1}. From the other aspect, the matrix $B$ is strictly diagonally dominant at $\mathbf{x}=[\frac{3}{2},\frac{9}{4},-\frac{19}{48},-\frac{1}{6},\frac{1}{12}]^T,$ with its diagonal elements being negative. Hence, $\mathbf{x}=[\frac{3}{2},\frac{9}{4},-\frac{19}{48},-\frac{1}{6},\frac{1}{12}]^T$ satisfies the conditions in Theorem \ref{res_1} by which there is a $\delta^*>0$ such that for each $\delta\in (0,\delta^*)$, the Nash equilibrium is exponentially stable.

The players' actions generated by the seeking strategy in  \eqref{cons_xx}-\eqref{NE_ave_cons_1}\footnote{Note that in the simulations, the distributed control gain discussed in Remark \ref{remar_gain} is included in the seeking strategy.}are plotted in Fig. \ref{Playeraction}. The initial values for the variables are set as  $\mathbf{x}(0)=[1\ \ 2\ \ 0\ \  0\ \ 0]^T$ and $\mathbf{y}_i(0)=[1\ \ 2\ \ 0\ \ 0\ \ 0]^T, \forall i\in \{1,2,\cdots,5\}$, which are close to $\mathbf{x}^*$ as only local convergence to the Nash equilibrium is ensured by Theorem \ref{res_1}. From the simulation result, it can be seen that the players' actions generated by the proposed seeking strategy converge to the Nash equilibrium point under the given initial conditions, which verifies Theorem \ref{res_1}.

\begin{figure}
\begin{center}
\scalebox{0.6}{\includegraphics[0,0][397,285]{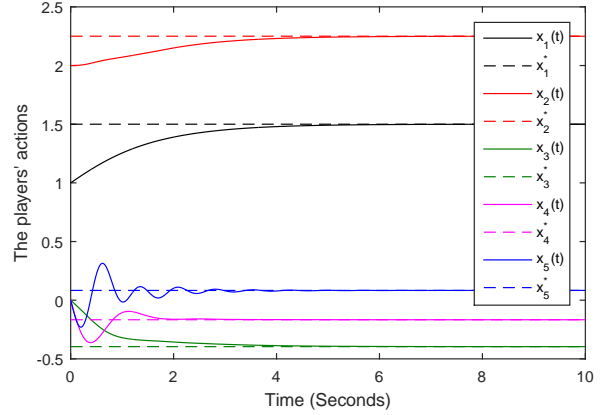}}
\caption{The plot of $x_{i}(t),$ $i\in \{1,2,3,4,5\}$ produced by \eqref{cons_xx}-\eqref{NE_ave_cons_1} in Example \ref{exam_1}.}\label{Playeraction}
\end{center}
\end{figure}

%



\begin{example}\label{exam_2}
The payoff function for player $i$ is
\begin{equation}
f_i(\mathbf{x})=m_if(\mathbf{x})+d_i,i\in\{1,2,\cdots,5\},
\end{equation}
where $m_1=1,m_2=5,m_3=2,m_4=3,m_5=2,d_i=0,\forall i\in \{1,2,\cdots,5\}$ in the simulation and
\begin{equation*}
\begin{aligned}
f(\mathbf{x})=&-\left(\frac{1}{12}x_1^4+5x_1^2+2x_1x_2+5x_2^2+x_2x_3\right.\\
&\left.+x_2x_5+\frac{5}{2}x_3^2+x_3x_4 +5x_4^2+2x_4x_5+3x_5^2\right).\\
\end{aligned}
\end{equation*}
The function $f(\mathbf{x})$ is at its maximum when $\mathbf{x}=\mathbf{0}_5$, which is also the Nash equilibrium of the game. Note that in this game, all the players' payoffs are maximized if $f(\mathbf{x})$ is maximized. Hence, any deviation from the Nash equilibrium reduces all the players' payoffs, indicating that adopting the proposed seeking strategy to seek for the Nash equilibrium would benefit all the players.
\end{example}

By direct calculation, it can be verified that in this game, Assumption \ref{Ass_4} is satisfied. Hence, for any bounded initial condition, there is a $\delta^*>0$ such that for each $\delta \in (0,\delta^*)$, the players' actions generated by the proposed method converge to the unique Nash equilibrium by Theorem \ref{corol_2}.
In the simulation, the initial values of the variables are set as $x_i(0)=20,y_{ij}(0)=20,\forall i,j\in \{1,2,\cdots,5\},$ which are far away from the equilibrium point. The players' actions generated by the proposed method in \eqref{cons_xx}-\eqref{NE_ave_cons_1} are shown in Fig. \ref{Playeraction_qua_2}. It can be seen from the simulation result that the players' actions generated by the proposed method  in \eqref{cons_xx}-\eqref{NE_ave_cons_1} converge to the Nash equilibrium though the initial values of the variables are far away from the equilibrium point, which verifies Theorem \ref{corol_2}.

\begin{figure}
\begin{center}
\scalebox{0.6}{\includegraphics[0,0][397,285]{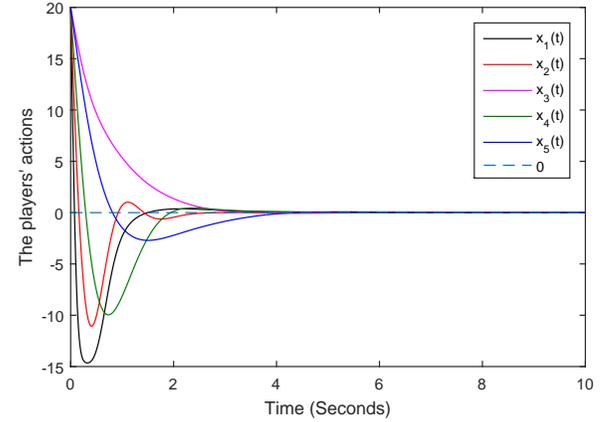}}
\caption{The plot of $x_{i}(t),$ $i\in \{1,2,3,4,5\}$ produced by \eqref{cons_xx}-\eqref{NE_ave_cons_1} in Example \ref{exam_2}.}\label{Playeraction_qua_2}
\end{center}
\end{figure}


\subsection{A quadratic game}

\begin{example}\label{exam_3}
The payoff function of player $i$ is
\begin{equation*}
f_i(\mathbf{x})=-\rho_i(x_i-x_i^d)^2-(p_0\sum_{i=1}^N x_i+q_0)x_i,
\end{equation*}
for $i\in\{1,2,\cdots,5\},$ where $\rho_i>0,x_i^d$ for  $i\in\{1,2,\cdots,5\},$ $p_0>0$ and $q_0$  are constants.
The given example can be utilized to model the energy consumption game for heating ventilation and air conditioning systems (see, e.g., \cite{YE15} and the references therein).
\end{example}

In the simulation, the parameters are set as $\rho_i=1, $ for $i\in \{1,2,\cdots,5\},$ $p_0=0.1, q_0=10$ and $x_i^d$ for $i\in \{1,2,\cdots,5\}$ are set as $10, 15, 20, 25, 30,$ respectively.  By direct calculation, it can be derived that the unique Nash equilibrium is $\mathbf{x}^*=[ 2.0147, 6.7766, 11.5385,16.3004, 21.0623]^T.$ For this example, $H$ defined in Assumption \ref{Assum_5} is strictly diagonally dominant with all the diagonal elements being negative. Hence, the conditions in Corollary \ref{qua_a} are satisfied by which, it can be concluded that there is a $\delta^*>0$ such that for each $\delta\in (0,\delta^*),$ the equilibrium is globally exponentially stable under the given strategy. In the simulation, the initial conditions of  all the variables in \eqref{cons_xx}-\eqref{NE_ave_cons_1} are set as $-10$.  The players' actions generated by the proposed method in \eqref{cons_xx}-\eqref{NE_ave_cons_1} are plotted in Fig. \ref{Playeraction_qua}. From the simulation result, it can be seen that though the initial conditions are far away from the equilibrium, the players' actions still converge to the Nash equilibrium, which verifies Corollary \ref{qua_a}.

\begin{figure}
\begin{center}
\scalebox{0.6}{\includegraphics[0,0][397,275]{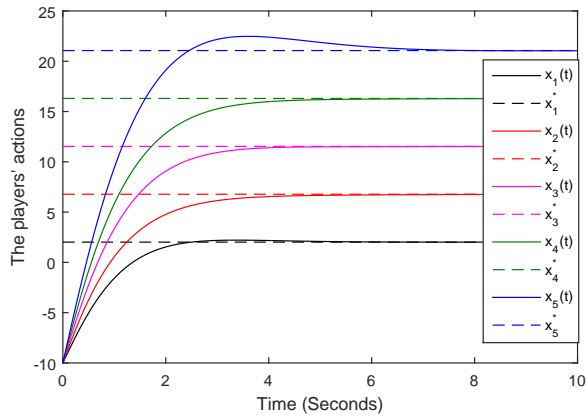}}
\caption{The plot of $x_{i}(t),$ $i\in \{1,2,3,4,5\}$ produced by \eqref{cons_xx}-\eqref{NE_ave_cons_1} in Example \ref{exam_3}.}\label{Playeraction_qua}
\end{center}
\end{figure}




\section{Conclusion} \label{concl_a}
Distributed Nash equilibrium seeking by neighboring communication for non-cooperative games among a network of players is studied in this paper. The players are supposed to be equipped with an undirected and connected communication graph. Based on a leader-following consensus protocol and the gradient play, a Nash equilibrium seeking algorithm is designed. For non-cooperative games, local convergence to the Nash equilibrium is firstly provided under mild conditions. Then, non-local convergence results are derived for the non-quadratic games under stronger conditions. For quadratic games, it is proven that under the proposed seeking strategy, the Nash equilibrium is globally exponentially stable under certain conditions.

\section{Appendix}
\subsection{Basics on Graph Theory}\label{gra}

For a graph defined as $\mathcal{G}=(\mathcal{V},\mathcal{E}),$ where $\mathcal{E}$ is the edge set satisfying
$\mathcal{E}\subseteq \mathcal{V}\times \mathcal{V}$ with $\mathcal{V}=\{1,2,\cdots,N\}$ being the set of nodes in the
network, it is undirected if for every $(i,j)\in \mathcal{E},$ $(j,i)\in
\mathcal{E}.$ Furthermore, $j$ is a neighbor of agent $i$ if $(j,i)\in \mathcal{E}.$ An undirected graph is connected if there is a path between any pair
of distinct vertices. The element on the $i$th row and $j$th column of the adjacency matrix $\mathcal{A}$ is defined as
$a_{ij}=1$ if node $j$ is connected with node $i,$ else, $a_{ij}=0.$ Moreover, we suppose that $a_{ii}=0$ in this paper.
The Laplacian matrix for the graph $L$ is defined as $L=\mathcal{D}-\mathcal{A} ,$ where $\mathcal{D}$ is a diagonal matrix whose $i$th diagonal entry is equal to the out
degree of node $i,$ represented by ${\displaystyle\sum\limits_{j=1}^{N}}a_{ij}$ \cite{HuSCL12}.

\subsection{Proof of Theorem \ref{res_1} }\label{res_1_p}
To facilitate the subsequent analysis, define an auxiliary system as
\begin{equation}\label{proof_aux_1}
\dot{x}_i=\bar{k}_i\frac{\partial f_i(\mathbf{x})}{\partial x_i}, i\in \mathcal{N}.
\end{equation}
Linearizing \eqref{proof_aux_1} at $\mathbf{x}^*$ gives
\begin{equation}
\frac{d\mathbf{x}}{dt}=\bar{\mathbf{k}}B (\mathbf{x}-\mathbf{x}^*),
\end{equation}
where $\bar{\mathbf{k}}=\text{diag}\{\bar{k}_i\}, i\in \mathcal{N}.$
Since $B$ is strictly diagonally dominant with all the diagonal elements being negative, $\bar{\mathbf{k}}B$ is Hurwitz by the Gershgorin Circle Theorem \cite{Horn85}. Hence, the equilibrium point is exponentially stable under \eqref{proof_aux_1}.  Therefore, there is a function $W_1:D_0\rightarrow R$, where $D_0=\{\mathbf{x} \in R^N| ||\mathbf{x}-\mathbf{x}^*||\leq r_0\}$ for some positive constant $r_0$ such that
\begin{equation}\label{proof_aux_2}
\begin{aligned}
&c_1||\mathbf{x}-\mathbf{x}^*||^2 \leq W_1(\mathbf{x})\leq c_2||\mathbf{x}-\mathbf{x}^*||^2\\
&\left(\frac{\partial W_1(\mathbf{x})}{\partial \mathbf{x}}\right)^T \left(\bar{\mathbf{k}} \frac{\partial G(\mathbf{x})}{\partial \mathbf{x}}\right)\leq -c_3 ||\mathbf{x}-\mathbf{x}^*||^2\\
&\left|\left| \frac{\partial W_1(\mathbf{x})}{\partial \mathbf{x}}\right|\right|\leq c_4 ||\mathbf{x}-\mathbf{x}^*||,
\end{aligned}
\end{equation}
for some positive constants $c_1,c_2,c_3,c_4$ by the Lyapunov Converse Theorem \cite{KHAIL02}.

Define
\begin{equation}\label{def_err}
\mathbf{\bar{y}}=\mathbf{y}-\mathbf{y}^q,
\end{equation}
where $\mathbf{y}^q=[y_{11}^q,y_{12}^q,\cdots,y_{1N}^q,y_{21}^q,\cdots,y_{2N}^q, \cdots, y_{NN}^q]^T$ and $y_{ij}^q=x_j,\forall i,j\in \mathcal{N}.$
Then,
\begin{equation}
\begin{aligned}
\dot{\bar{\mathbf{y}}}&=\dot{\mathbf{y}}-\dot{\mathbf{y}}^q\\
&=-((L\otimes I_{N\times N}+B_0)(\bar{\mathbf{y}}+\mathbf{y}^q)+h_1(\mathbf{x}))-\dot{\mathbf{y}}^q\\
&=-(L\otimes I_{N\times N}+B_0)\bar{\mathbf{y}}-\dot{\mathbf{y}}^q.
\end{aligned}
\end{equation}

Define a Lyapunov candidate function as
\begin{equation}
V=cW_1(\mathbf{x})+(1-c)\bar{\mathbf{y}}^TP_1\bar{\mathbf{y}},
\end{equation}
where $c\in (0,1)$ is a constant and $P_1$ is a symmetric positive definite matrix such that
\begin{equation}\label{def_p_1}
P_1(L\otimes I_{N\times N}+B_0)+(L\otimes I_{N\times N}+B_0)^TP_1=Q_1,
\end{equation}
where $Q_1$ is a symmetric positive definite matrix as $-(L\otimes I_{N\times N}+B_0)$ is Hurwitz by noticing that the communication graph is undirected and connected.

Define $\mathbf{z}=[(\mathbf{x}-\mathbf{x}^*)^T, \bar{\mathbf{y}}^T]^T.$ Then, there exists a domain $D_1=\{\mathbf{z}\in R^{N+N^2}| ||\mathbf{z}||\leq r_1\},$ for some positive constant $r_1$ such that the time derivative of the Lyapunov candidate function satisfies
\begin{equation}
\begin{aligned}
\dot{V}=&c\delta \left(\frac{\partial W_1(\mathbf{x})}{\partial \mathbf{x}}\right)^T \left(\bar{\mathbf{k}} \frac{\partial G(\mathbf{x})}{\partial \mathbf{x}}\right)\\
&+c\delta \left(\frac{\partial W_1(\mathbf{x})}{\partial \mathbf{x}}\right)^T \bar{\mathbf{k}} \left(\frac{\partial G}{\partial \mathbf{x}}(\mathbf{y})-\frac{\partial G}{\partial \mathbf{x}}(\mathbf{x})\right)\\
&+(1-c)(-(L\otimes I_{N\times N}+B_0)\bar{\mathbf{y}}-\dot{\mathbf{y}}^q)^TP_1\bar{\mathbf{y}}\\
&+(1-c)\bar{\mathbf{y}}^TP_1(-(L\otimes I_{N\times N}+B_0)\bar{\mathbf{y}}-\dot{\mathbf{y}}^q)\\
\leq& -c\delta c_3 ||\mathbf{x}-\mathbf{x}^*||^2-(1-c)\bar{\mathbf{y}}^TQ_1\bar{\mathbf{y}}\\
&-2(1-c)\bar{\mathbf{y}}^TP_1 \dot{\mathbf{y}}^q+\delta l_1||\mathbf{x}-\mathbf{x}^*||||\bar{\mathbf{y}}||,
\end{aligned}
\end{equation}
for some positive constant $l_1$ by noticing that the $\frac{\partial f_i}{\partial x_i}(\mathbf{x})$ for $i\in \{1,2,\cdots,N\}$ are Lipschitz.
Let $\lambda_{min}(Q_1)$ denote the minimal eigenvalue of $Q_1$. Then,
\begin{equation}
\begin{aligned}
\dot{V}\leq &-\delta c c_3 ||\mathbf{x}-\mathbf{x}^*||^2-(1-c)\lambda_{min}(Q_1)||\bar{\mathbf{y}}||^2\\
&+\delta l_1 ||\mathbf{x}-\mathbf{x}^*||||\bar{\mathbf{y}}||-2(1-c)\bar{\mathbf{y}}^TP_1 \left(\frac{\partial (\mathbf{y}^q)^T}{\partial \mathbf{x}}\right)^T\frac{d \mathbf{x}}{dt}\\
\leq & -\delta c c_3 ||\mathbf{x}-\mathbf{x}^*||^2-(1-c)\lambda_{min}(Q_1)||\bar{\mathbf{y}}||^2\\
&+\delta l_2 ||\mathbf{x}-\mathbf{x}^*||||\bar{\mathbf{y}}||+\delta l_3 ||\bar{\mathbf{y}}||^2,
\end{aligned}
\end{equation}
for some positive constants $l_2$ and $l_3$ by noticing that $\left|\left|\frac{dx_i}{dt}\right|\right|=\left|\left|\delta \bar{k}_i \frac{\partial f_i}{\partial x_i}(\mathbf{y}_i)\right|\right|
=\left|\left|\delta \bar{k}_i \frac{\partial f_i}{\partial x_i}(\mathbf{y}_i)-\delta \bar{k}_i \frac{\partial f_i}{\partial x_i}(\mathbf{x}^*)\right|\right|
\leq  \delta \bar{l}_{i1} ||\bar{\mathbf{y}}||+ \delta \bar{l}_{i2} ||\mathbf{x}-\mathbf{x}^*||,$
for some positive constants $\bar{l}_{i1}$ and $\bar{l}_{i2},$ for $i\in \mathcal{N}.$

Define $B_1=\left[
      \begin{array}{cc}
        \frac{(1-c)\lambda_{min}(Q_1)}{\delta}-l_3 & -\frac{l_2}{2} \\
        -\frac{l_2}{2} & cc_3 \\
      \end{array}
    \right].$
 Then, if $ 0<\delta<\frac{4(1-c)cc_3\lambda_{min}(Q_1)}{l_2^2+4cc_3l_3},$
 matrix $B_1$ is symmetric positive definite.
 If this is the case, $ \dot{V}\leq -\delta \lambda_{min}(B_1)||\mathbf{z}||^2,$
 where $\lambda_{min}(B_1)$ is the minimal eigenvalue of $B_1$ and $\lambda_{min}(B_1)>0.$

 Noticing that there exist positive constants $\bar{c}_1$ and $\bar{c}_2$ such that $ \bar{c}_1||\mathbf{z}||^2\leq V \leq  \bar{c}_2||\mathbf{z}||^2,$
 it can be derived that
 \begin{equation}
 ||\mathbf{z}(t)||\leq \sqrt{\frac{\bar{c}_2}{\bar{c}_1}}e^{-\frac{\delta \lambda_{min}(B_1)}{2\bar{c}_2}t}||\mathbf{z}(0)||,
 \end{equation}
 by utilizing the Comparison Lemma \cite{KHAIL02}.

 Furthermore, define $\mathbf{E}_r(t)=[(\mathbf{x}-\mathbf{x}^*)^T,(\mathbf{y}-\mathbf{y}^q(\mathbf{x}^*))^T]^T.$ Then, $||\mathbf{E}_r(t)||\leq  ||\mathbf{z}(t)||+||\mathbf{y}^q(\mathbf{x})-\mathbf{y}^q(\mathbf{x}^*)||
\leq K_1 ||\mathbf{z}(t)||\leq K_1\sqrt{\frac{\bar{c}_2}{\bar{c}_1}}e^{-\frac{\delta \lambda_{min}(B_1)}{2\bar{c}_2}t}||\mathbf{z}(0)||
\leq  K_1\sqrt{\frac{\bar{c}_2}{\bar{c}_1}}e^{-\frac{\delta \lambda_{min}(B_1)}{2\bar{c}_2}t}(||\mathbf{E}_r(0)||+||\mathbf{y}^q(\mathbf{x}^*)-\mathbf{y}^q(\mathbf{x}(0))||)
\leq K_2\sqrt{\frac{\bar{c}_2}{\bar{c}_1}}e^{-\frac{\delta \lambda_{min}(B_1)}{2\bar{c}_2}t}||\mathbf{E}_r(0)||,$
 for some positive constants $K_1$ and $K_2$.
 Hence, the conclusion is derived.

\subsection{Proof of Theorem \ref{corol_2}}\label{corol_2_p}
Define the Lyapunov candidate function as
\begin{equation}
V=\frac{c}{2}(\mathbf{x}-\mathbf{x}^*)^T\bar{\mathbf{k}}^{-1}(\mathbf{x}-\mathbf{x}^*)+(1-c)\bar{\mathbf{y}}^TP_1\bar{\mathbf{y}},
\end{equation}
where $c\in (0,1)$ is a constant, $\bar{\mathbf{k}}=\text{diag}\{\bar{k}_i\},i\in \mathcal{N},$ $P_1$ is defined in \eqref{def_p_1} and $\bar{\mathbf{y}}$ is defined in \eqref{def_err}.
Then, there exist positive constants $\bar{c}_1$ and $\bar{c}_2$ such that $\bar{c}_1||\mathbf{z}||^2\leq V\leq \bar{c}_2||\mathbf{z}||^2,$
where $\mathbf{z}=[(\mathbf{x}-\mathbf{x}^*)^T,\bar{\mathbf{y}}^T]^T$.
Furthermore, for any $(\mathbf{x},\mathbf{y})$ that belongs to a compact set, the time derivative of the Lyapunov candidate function satisfies
\begin{equation}\label{comp_set_1}
\begin{aligned}
\dot{V}=& \delta c(\mathbf{x}-\mathbf{x}^*)^T\frac{\partial G}{\partial \mathbf{x}}(\mathbf{y})\\
&+(1-c)\dot{\bar{\mathbf{y}}}^TP_1\bar{\mathbf{y}}+(1-c)\bar{\mathbf{y}}^TP_1\dot{\bar{\mathbf{y}}}\\
=&\delta c(\mathbf{x}-\mathbf{x}^*)^T\frac{\partial G}{\partial \mathbf{x}}(\mathbf{x})+\delta c (\mathbf{x}-\mathbf{x}^*)^T\left(\frac{\partial G}{\partial \mathbf{x}}(\mathbf{y})-\frac{\partial G}{\partial \mathbf{x}}(\mathbf{x})\right)\\
&+(1-c)(-(L\otimes I_{N\times N}+B_0)\bar{\mathbf{y}}-\dot{\mathbf{y}}^q)^TP_1\bar{\mathbf{y}}\\
&+(1-c)\bar{\mathbf{y}}^TP_1(-(L\otimes I_{N\times N}+B_0)\bar{\mathbf{y}}-\dot{\mathbf{y}}^q)\\
\leq &-\delta cm||\mathbf{x}-\mathbf{x}^*||^2+\delta cl_1||\mathbf{x}-\mathbf{x}^*||||\bar{\mathbf{y}}||\\
&-(1-c)\bar{\mathbf{y}}^TQ_1\bar{\mathbf{y}}-2(1-c)\bar{\mathbf{y}}^TP_1\dot{\mathbf{y}}^q,
\end{aligned}
\end{equation}
for some positive constant $l_1$.

Hence, for any $(\mathbf{x},\mathbf{y})$ that belongs to the compact set,
\begin{equation}\label{comp_set}
\begin{aligned}
\dot{V}\leq & -\delta cm||\mathbf{x}-\mathbf{x}^*||^2+\delta cl_1||\mathbf{x}-\mathbf{x}^*||||\bar{\mathbf{y}}||\\
&-(1-c)\bar{\mathbf{y}}^TQ_1\bar{\mathbf{y}}-2(1-c)\bar{\mathbf{y}}^TP_1\left(\frac{\partial (\mathbf{y}^q)^T}{\partial \mathbf{x}}\right)^T\frac{d\mathbf{x}}{dt}\\
\leq &  -\delta cm||\mathbf{x}-\mathbf{x}^*||^2-(1-c)\lambda_{min}(Q_1)||\bar{\mathbf{y}}||^2\\
&+\delta cl_2||\mathbf{x}-\mathbf{x}^*||||\bar{\mathbf{y}}||+\delta l_3 ||\bar{\mathbf{y}}||^2,
\end{aligned}
\end{equation}
for some positive constants $l_2$ and $l_3$, in which $\lambda_{min}(Q_1)$ is the minimal eigenvalue of $Q_1$. The second inequality in \eqref{comp_set} is derived by using the fact that for $(\mathbf{x},\mathbf{y})$ that belongs to the compact set, $\left|\left|\frac{\partial f_i}{\partial x_i}(\mathbf{y}_i)\right|\right|=\left|\left|\frac{\partial f_i}{\partial x_i}(\mathbf{y}_i)-\frac{\partial f_i}{\partial x_i}(\mathbf{x}^*)\right|\right|
\leq \left|\left|\frac{\partial f_i}{\partial x_i}(\mathbf{y}_i)-\frac{\partial f_i}{\partial x_i}(\mathbf{x})\right|\right|+\left|\left|\frac{\partial f_i}{\partial x_i}(\mathbf{x})-\frac{\partial f_i}{\partial x_i}(\mathbf{x}^*)\right|\right|
\leq \bar{l}_{i1}||\bar{\mathbf{y}}||+\bar{l}_{i2}||\mathbf{x}-\mathbf{x}^*||,$
for some positive constants $ \bar{l}_{i1}$ and  $\bar{l}_{i2},$ for all $i\in \mathcal{N}.$

The rest of the proof follows the proof of Theorem 1 and is omitted here.

\subsection{Proof of Corollary \ref{corol_2_add}}\label{corol_2_add_p}
The proof follows the proof of Theorem \ref{corol_2}  by further noticing that \eqref{comp_set_1} and \eqref{comp_set} hold for any $\mathbf{z}\in R^{N+N^2}$ given that the functions $\frac{\partial f_i(\mathbf{x})}{\partial x_i},\forall i\in \mathcal{N},$ are globally Lipschitz.

\subsection{Proof of Corollary \ref{qua_a}}\label{qua_a_p}
If Assumption \ref{Assum_5} is satisfied, all the conditions in  Theorem \ref{res_1} are satisfied for the quadratic games. In addition, the Nash equilibrium is unique for the quadratic game. By the result in Theorem \ref{res_1}, there exists a $\delta^*>0$ such that for each $\delta\in (0,\delta^*)$, the equilibrium is exponentially stable. Moreover, the system in  \eqref{cons_xx}-\eqref{NE_ave_cons_1} is a linear system, by which it can be derived that the equilibrium is globally exponentially stable.

\subsection{Proof of Corollary \ref{qua_ca}}\label{qua_ca_p}

Note that for potential games $\frac{\partial f_i(\mathbf{x})}{\partial x_i}=\frac{\partial F(\mathbf{x})}{\partial x_i},\forall i\in \mathcal{N},$
which indicates that $\frac{\partial^2 f_i(\mathbf{x})}{\partial x_i \partial x_j}=\frac{\partial^2 F(\mathbf{x})}{\partial x_i \partial x_j},\forall i,j\in \mathcal{N}.$
Hence, $\frac{\partial^2 f_i(\mathbf{x})}{\partial x_i \partial x_j}=\frac{\partial^2 f_j(\mathbf{x})}{\partial x_i \partial x_j}, \forall i,j\in \mathcal{N}.$
Therefore, $H$ is a symmetric negative definite matrix.

Before we facilitate the subsequent analysis, we firstly show that $\bar{\mathbf{k}}H$ is Hurwitz.
Define an auxiliary system as
\begin{equation}\label{auxi_2}
\dot{\Phi}=\bar{\mathbf{k}}H\Phi,
\end{equation}
where $\bar{\mathbf{k}}=\text{diag}\{\bar{k}_i\},i\in \mathcal{N}.$
Let $\Phi=\sqrt{\bar{\mathbf{k}}}\Psi,$ where $\sqrt{\bar{\mathbf{k}}}$ is a diagonal matrix whose $i$th diagonal element is $\sqrt{\bar{k}_i}.$ Then,
\begin{equation}\label{auxi_1}
\dot{\Psi}=\sqrt{\bar{\mathbf{k}}}H\sqrt{\bar{\mathbf{k}}}\Psi.
\end{equation}
Since $\Psi^T\sqrt{\bar{\mathbf{k}}}H\sqrt{\bar{\mathbf{k}}}\Psi=\Phi^TH\Phi<0$ for every $||\Phi||\neq 0,$ hence for every   $||\Psi||\neq 0,$ it can be derived that,
$\sqrt{\bar{\mathbf{k}}}H\sqrt{\bar{\mathbf{k}}}$ is Hurwitz, which indicates that the equilibrium point of \eqref{auxi_1} is exponentially stable. Hence, the equilibrium point of \eqref{auxi_2} is exponentially stable, by which it can be derived that $\bar{\mathbf{k}}H$ is Hurwitz.

Noticing that $\bar{\mathbf{k}}H$ is Hurwitz, the Nash equilibrium is exponentially stable under the gradient play in \eqref{proof_aux_1}. The rest of the proof follows the proof of Theorem \ref{res_1} by further noticing that the proof of Theorem  \ref{res_1} holds for $\mathbf{z}\in R^{N+N^2}$ for the quadratic potential games.

\end{document}